\documentclass[a4paper,11pt]{article}
\usepackage{amsmath, amsfonts, amscd, amssymb, amsthm, enumerate, xypic}

\def\defterm{\emph}

\newcommand{\Mat}{\operatorname{M}}

\newcommand{\NT}{\operatorname{NT}}
\newcommand{\LT}{\operatorname{LT}}
\newcommand{\GL}{\operatorname{GL}}
\newcommand{\Ker}{\operatorname{Ker}}
\newcommand{\Vect}{\operatorname{span}}

\newcommand{\tr}{\operatorname{tr}}
\renewcommand{\setminus}{\smallsetminus}

\def\F{\mathbb{F}}
\def\K{\mathbb{K}}

\def\calU{\mathcal{U}}
\def\calV{\mathcal{V}}
\def\calW{\mathcal{W}}

\def\lcro{\mathopen{[\![}}
\def\rcro{\mathclose{]\!]}}

\theoremstyle{definition}
\newtheorem{Def}{Definition}

\theoremstyle{plain}
\newtheorem{theo}{Theorem}

\newtheorem{cor}[theo]{Corollary}
\newtheorem{lemme}[theo]{Lemma}
\newtheorem{claim}{Claim}

\theoremstyle{plain}

\theoremstyle{remark}
\newtheorem{Rems}{Remarks}
\newtheorem{Rem}[Rems]{Remark}

\title{On Gerstenhaber's theorem for spaces of nilpotent matrices over a skew field}
\author{Cl\'ement de Seguins Pazzis\footnote{Universit\'e de Versailles Saint-Quentin-en-Yvelines, Laboratoire de Math\'ematiques
de Versailles, 45 avenue des Etats-Unis, 78035 Versailles cedex, France}
\footnote{e-mail address: dsp.prof@gmail.com}}

\begin{document}

\thispagestyle{plain}

\maketitle

\begin{abstract}
Let $\K$ be a skew field, and $\K_0$ be a subfield of the central subfield of $\K$
such that $\K$ has finite dimension $q$ over $\K_0$.
Let $\calV$ be a $\K_0$-linear subspace of $n \times n$ nilpotent matrices with entries in $\K$.
We show that the dimension of $\calV$ is bounded above by $q\,\dbinom{n}{2}$, and that equality occurs if and only if
$\calV$ is similar to the space of all $n \times n$ strictly upper-triangular matrices over $\K$.
This generalizes famous theorems of Gerstenhaber and Serezhkin, which cover the special case $\K=\K_0$.
\end{abstract}



\section{Introduction}

In this article, we let $\K$ be an arbitrary skew field, and $\K_0$ be a subfield of the central subfield of $\K$
over which $\K$ has finite dimension $q$. The set $\K^n$ is always endowed with its canonical structure of right-$\K$-vector space.
We denote by $\Mat_{n,p}(\K)$ the set of all $n \times p$ matrices with entries in $\K$, endowed with its canonical structure of vector space
over $\K_0$. We set $\Mat_n(\K):=\Mat_{n,n}(\K)$, and denote by $\GL_n(\K)$ its group of invertible elements. We denote
by $\NT_n(\K)$ the set of all strictly upper-triangular matrices of $\Mat_n(\K)$.

The transpose of a matrix $M$ is denoted by $M^T$, and its trace by $\tr(M)$.
The relation of similarity between matrices is denoted by $\simeq$ and is naturally extended to subsets of $\Mat_n(\K)$.

\vskip 3mm
A linear subspace $\calV$ of $\Mat_n(\K)$ (over $\K_0$) is called \defterm{nilpotent} when all its elements are nilpotent matrices.
In that case, we note that, for every $P \in \GL_n(\K)$, the set $P\calV P^{-1}$ is a nilpotent linear subspace of $\Mat_n(\K)$
with the same dimension as $\calV$.

In his first entry in a series of four landmark papers \cite{Ger1}, Murray Gerstenhaber
studied the structure of such nilpotent subspaces. Here is his most famous result:

\begin{theo}[Gerstenhaber, Serezhkin]\label{Gerstenhabertheorem}
Assume that $\K$ is commutative, and let $\calV$ be a nilpotent linear subspace of the $\K$-vector space $\Mat_n(\K)$. Then
$\dim_\K \calV \leq \dbinom{n}{2}$, and equality occurs if and only if $\calV$ is similar to $\NT_n(\K)$.
\end{theo}

Our main aim here is to prove the following generalization to skew fields:

\begin{theo}\label{Gerstenhaberskew}
Let $\calV$ be a nilpotent linear subspace of $\Mat_n(\K)$ (over $\K_0$). Then:
\begin{enumerate}[(a)]
\item $\dim_{\K_0} \calV \leq q\,\dbinom{n}{2}$.
\item If $\dim_{\K_0} \calV = q\,\dbinom{n}{2}$, then $\calV$ is similar to $\NT_n(\K)$.
\end{enumerate}
\end{theo}

If $\K$ is finite (and therefore commutative), choosing $\K_0$ as its prime subfield yields the following corollary:

\begin{cor}\label{finitefieldscor}
Assume $\K$ is finite with cardinality $p$.
Let $\calV$ be a subgroup of $(\Mat_n(\K),+)$ in which every matrix is nilpotent.
Then $\# \calV \leq p^{\binom{n}{2}}$, and equality occurs only if $\calV$ is similar to $\NT_n(\K)$.
\end{cor}

At the time of \cite{Ger1}, Gerstenhaber was actually able to prove Theorem \ref{Gerstenhabertheorem} only for fields with at least $n$ elements,
mostly because his methods relied on the use of polynomials. A lot of progress has been made
since then: we now have elementary and elegant proofs of the inequality statement that are valid for every field
\cite{Mathes,MacD}, and the case of equality has been obtained for an arbitrary field by V.N. Serezhkin
\cite{Serezhkin} (for fields with more than two elements, we now have a shorter proof based upon Jacobson's generalization of Engels's theorem, see
\cite{Mathes}).

Recent progress on the topic must be signaled here: in \cite{dSPlargerank}, the inequality statement of Theorem \ref{Gerstenhabertheorem}
has been extended to linear subspaces of $\Mat_n(\K)$ \defterm{with a trivial spectrum}, i.e.,
which consist solely of matrices with no non-zero eigenvalue in $\K$.
The study of such spaces is motivated by its connection with the affine subspaces of matrices with a rank bounded below by some fixed integer.
More recently \cite{dSPaffinenonsingular}, a classification of the
linear subspaces of $\Mat_n(\K)$ with a trivial spectrum and the maximal dimension $\binom{n}{2}$ has been discovered
for fields with more than two elements: for such fields, Theorem \ref{Gerstenhabertheorem} appears as an easy consequence of it
(see Section 5 of \cite{dSPaffinenonsingular}).
Finally, in \cite{dSPsoleeigenvalue}, we have been able to prove a theorem similar to Gerstenhaber's
for linear subspaces of matrices with exactly one eigenvalue in an algebraic closure of $\K$.

Both \cite{dSPaffinenonsingular} and \cite{dSPsoleeigenvalue} are based upon a new technique which we will call
the \emph{diagonal-compatibility method}. The purpose of this paper is to demonstrate how this strategy can be used
to obtain Theorem \ref{Gerstenhaberskew} with essentially no prior knowledge on the topic.
In particular, this will yield an alternative proof of Theorem \ref{Gerstenhabertheorem}
(in the course of the proof, we will point out to some shortcuts for the case $\K=\K_0$).
Note that in some cases (e.g., $\K$ is commutative and separable over $\K_0$), the line of reasoning of \cite{Mathes} may be adapted with some effort
by using the trace of $\K$ over $\K_0$; this however fails to yield our more general theorem, so we will not
use this strategy.

Our key lemma, which is proven in Section \ref{keylemma}, is a variation of Proposition 10 of \cite{dSPlargerank}.
It will help us prove both points in Theorem \ref{Gerstenhaberskew}: first, point (a) in Section \ref{inequalityproof}
and then point (b) in the longer Section \ref{equalityproof}.

For to simplify the case $\K=\K_0$, we recall the following classical result, which is proven in \cite{MacD,Mathes}.
We give a simple proof of it.

\begin{lemme}\label{orthogonality}
Assume that $\K$ is commutative, and let $A$ and $B$ be two nilpotent matrices of $\Mat_n(\K)$ such that $A+B$ is nilpotent. Then $\tr(AB)=0$.
\end{lemme}

\begin{proof}
For $M=(m_{i,j})_{1 \leq i,j \leq n}$, we denote by $c_2(M)$ the coefficient in front of $t^{n-2}$ in the characteristic polynomial of $M$.
Using $c_2(M)=\underset{1 \leq i<j \leq n}{\sum} \begin{vmatrix}
m_{i,i} & m_{i,j} \\
m_{j,i} & m_{j,j}
\end{vmatrix}$, one finds the formula
\begin{equation}
\forall (M,N) \in \Mat_n(\K)^2, \; c_2(M+N)-c_2(M)-c_2(N)=\tr(M)\tr(N)-\tr(MN).
\end{equation}
As $A$, $B$ and $A+B$ are nilpotent, we find $\tr(A)=\tr(B)=0$ and $c_2(A)=c_2(B)=c_2(A+B)=0$, which yields $\tr(AB)=0$.
\end{proof}

\section{The key lemma}\label{keylemma}

\begin{Def}
Let $\calV$ be a subset of $\Mat_n(\K)$.
A vector $X \in \K^n$ is called \defterm{$\calV$-adapted} if it is non-zero and
no matrix of $\calV$ has $X\K$ as its column space.
\end{Def}

\begin{lemme}\label{basiclemma}
Let $\calV$ be a subset of $\Mat_n(\K)$ which is closed under addition and contains only nilpotent matrices,
and denote by $(e_1,\dots,e_n)$ the canonical basis of the $\K$-vector space $\K^n$.
Then one of the vectors $e_1,\dots,e_n$ is $\calV$-adapted.
\end{lemme}

The proof is largely similar to that of Proposition 10 in \cite{dSPlargerank}.

\begin{proof}
The result is trivial for $n=1$. We use an induction, assuming, given an integer $n \geq 2$,
that the result holds for the integer $n-1$.
Let $\calV$ be a subset of $\Mat_n(\K)$ which is closed under addition and contains only nilpotent matrices.
We assume that none of $e_1,\dots,e_n$ is $\calV$-adapted. \\
For $(i,j) \in \lcro 1,n\rcro^2$, we denote by $E_{i,j}$ the matrix of $\Mat_n(\K)$ with a zero entry everywhere except at the
$(i,j)$-spot where the entry is $1$.
Denote by $\calW$ the subset of $\calV$ consisting of its matrices with a zero $n$-th row.
Every $M \in \calW$ may be written as
$$M=\begin{bmatrix}
K(M) & [?]_{(n-1) \times 1} \\
[0]_{1 \times (n-1)} & 0
\end{bmatrix} \quad \text{with $K(M) \in \Mat_{n-1}(\K)$,}$$
so that $K(\calW)$ consists of nilpotent matrices and is obviously closed under addition.
By induction, we know that there is some $i \in \lcro 1,n-1\rcro$ such that $e_i$ is $K(\calW)$-adapted
(identifying $\K^{n-1}$ with the subspace $\K^{n-1} \times \{0\}$ of $\K^n$ in the usual way).
However, we have assumed that $e_i$ is not $\calV$-adapted, therefore some matrix $M$ of $\calV$
has all rows zero except the $i$-th. Then $M \in \calW$, and as $e_i$ is $K(\calW)$-adapted, we find that $K(M)=0$.
Thus, $M=a\,E_{i,n}$ for some $a \in \K \setminus \{0\}$. \\
Now, the same argument may be applied to $P\,\calV\,P^{-1}$ for any $n \times n$ permutation matrix $P$.
By doing so, we find a map $f : \lcro 1,n\rcro \rightarrow \lcro 1,n\rcro$ and a list $(a_1,\dots,a_n) \in (\K \setminus \{0\})^n$
such that $\calV$ contains $a_k\,E_{f(k),k}$ for all $k \in \lcro 1,n\rcro$.
Let us choose a cycle for $f$, i.e.\ a list $(i_1,\dots,i_p)$ of pairwise distinct elements of $\lcro 1,n\rcro$
such that $f(i_1)=i_2,\dots,f(i_{p-1})=i_p$ and $f(i_p)=i_1$. To obtain such a cycle, one notes that
some element in the sequence $(f^i(1))_{i \geq 0}$ appears several times, to the effect that one may choose non-negative integers
$i<j$, with $j-i$ minimal, such that $f^i(1)=f^j(1)$; then
$(i_1,\dots,i_p):=(f^i(1),\dots,f^{j-1}(1))$ is a cycle for $f$.

Then, the matrix $M:=\underset{k=1}{\overset{p}{\sum}}a_{i_k} E_{f(i_k),i_k}$ belongs to $\calV$ and satisfies
$M^p e_{i_1}=e_{i_1}\,\Bigl(\underset{k=1}{\overset{p}{\prod}} a_{i_{p+1-k}}\Bigr)$. This shows that $M$ is non-nilpotent, which is a contradiction.
This \emph{reductio ad absurdum} yields that some $e_j$ is $\calV$-adapted, which concludes the proof by induction.
\end{proof}

\section{Proving the inequality statement}\label{inequalityproof}

Now, we use Lemma \ref{basiclemma} to obtain point (a) of Theorem \ref{Gerstenhaberskew}, just as Proposition 10 was used
to obtain Theorem 9 in \cite{dSPlargerank}.

Again, we use an induction on $n$. The case $n=1$ is trivial.
Let $\calV$ be a nilpotent linear subspace of the $\K_0$-vector space $\Mat_n(\K)$.
First of all, we know that some $e_i$ is $\calV$-adapted.
Replacing $\calV$ with $P\,\calV\,P^{-1}$ for a well-chosen permutation matrix $P$, we may assume that
$e_n$ is $\calV$-adapted. In that case, we write every matrix of $\calV$ as
$$M=\begin{bmatrix}
K(M) & C(M) \\
L(M) & a(M)
\end{bmatrix},$$
where $K(M)$, $C(M)$, $L(M)$ are respectively $(n-1) \times (n-1)$, $(n-1) \times 1$, $1 \times (n-1)$ matrices, and $a(M) \in \K$.
Set
$$\calW_1:=\bigl\{M \in \calV : \; C(M)=0\bigr\}.$$
Any $M \in \calW_1$ is nilpotent, which yields that $a(M)=0$ and $K(M)$ is nilpotent.
Moreover, that $e_n$ is $\calV$-adapted yields:
$$\forall M \in \calW_1, \; K(M)=0 \Rightarrow M=0.$$
Using the rank theorem, one finds
$$\dim_{\K_0} \calV=\dim_{\K_0} K(\calW_1)+\dim_{\K_0} C(\calV).$$
As $K(\calW_1)$ is a nilpotent $\K_0$-linear subspace of $\Mat_n(\K)$ and $C(\calV) \subset \K^{n-1}$, the induction hypothesis yields
$$\dim_{\K_0} \calV \leq q\,\binom{n-1}{2}+q\,(n-1)=q\,\binom{n}{2}.$$
Thus, point (a) of Theorem \ref{Gerstenhaberskew} is proven by induction on $n$.

\section{Solving the case of equality}\label{equalityproof}

Here, we prove point (b) of Theorem \ref{Gerstenhaberskew} by induction on $n$.
The case $n=1$ is trivial.

\subsection{The case $n=2$}

This case is trivial if $\K=\K_0$ but otherwise needs an explanation.
Let $A,B$ be non-zero nilpotent matrices of $\Mat_2(\K)$ such that $A+B$ is nilpotent. Assume that $\Ker A \neq \Ker B$.
Then $\K^2=\Ker A \oplus \Ker B$, and we may therefore find a basis $(f_1,f_2)$ of the $\K$-vector space $\K^2$
such that $f_1 \in \Ker A$ and $f_2 \in \Ker B$. This yields some $P \in \GL_2(\K)$ and some $(a,b)\in (\K \setminus \{0\})^2$ such that
$$PAP^{-1}=\begin{bmatrix}
0 & a \\
0 & 0
\end{bmatrix} \quad \text{and} \quad PBP^{-1}=\begin{bmatrix}
0 & 0 \\
b & 0
\end{bmatrix}.$$
Therefore $P(A+B)P^{-1}=\begin{bmatrix}
0 & a \\
b & 0
\end{bmatrix}$, which is a non-singular matrix. This is a contradiction. \\
Now, let $\calV$ be a $q$-dimensional linear subspace of the $\K_0$-vector space $\Mat_2(\K)$ in which every matrix is nilpotent.
Choose $A \in \calV \setminus \{0\}$. Then we have just shown that every non-zero matrix of $\calV$ vanishes on $\Ker A$.
Choosing a basis $(g_1,g_2)$ of the $\K$-vector space $\K^2$ with $g_1$ in $\Ker A$, we find a non-singular matrix $P \in \GL_2(\K)$
such that every matrix of $P \calV P^{-1}$ has a zero first column. As $P\calV P^{-1}$ is nilpotent, we deduce that
$P \calV P^{-1} \subset \NT_2(\K)$, and the equality of dimensions yields $P \calV P^{-1}=\NT_2(\K)$.

\subsection{Setting things up for $n \geq 3$}\label{setup}

In the rest of the proof, we assume that $n \geq 3$ and that point (b) of Theorem \ref{Gerstenhaberskew} holds
for any nilpotent linear subspace of the $\K_0$-vector space $\Mat_{n-1}(\K)$.

Let $\calV$ be a nilpotent $\K_0$-linear subspace of $\Mat_n(\K)$ with dimension $q\dbinom{n}{2}$.
Seing $\calV$ as a set of linear endomorphisms of the right-$\K$-vector space $\K^n$,
what we need is to find a basis $(e'_1,\dots,e'_n)$ of the $\K$-vector space $\K^n$ in which the operators in $\calV$ are represented exactly by the strictly upper-triangular $n \times n$
matrices. Our method is to construct such a basis step-by-step. Equivalently, we will replace successively $\calV$
with similar linear subspace of matrices in order to simplify $\calV$ more and more, until we finally find the space $\NT_n(\K)$.
Let us quickly lay out the sequence of choices that we will make:
\begin{itemize}
\item We will start by choosing the last vector $e'_n$ among the vectors that are $\calV$-adapted.
Then we will choose a basis $(\overline{e'_1},\dots,\overline{e'_{n-1}})$ of the quotient space $\K^n/(e'_n \K)$ that is well-suited to $\calV$.
Those first two operations will be done within the current section.
\item At this point, each one of the vectors $e'_1,\dots,e'_{n-1}$ will be well determined \emph{up to addition of a vector of $e'_n\K$.}
\item A reasonable choice of $e'_2,\dots,e'_{n-1}$ will then be obtained (Section \ref{cornercompatsection}).
\item A reasonable choice of $e'_1$ will come last, after a more extensive inquiry (in the end of Section \ref{analysefetg}).
\end{itemize}

In the rest of the proof, we denote by $(e_1,\dots,e_n)$ the canonical basis of the $\K$-vector space $\K^n$.
As in Section \ref{inequalityproof}, we lose no generality in assuming that
$e_n$ is $\calV$-adapted. With the same notation as in Section \ref{inequalityproof}, we deduce from the
equality $\dim_{\K_0} \calV=q\,\binom{n}{2}$ that
$$\dim_{\K_0} K(\calW_1)=q\,\binom{n-1}{2} \quad \text{and} \quad \dim_{\K_0} C(\calV)=q\,(n-1).$$
Set
$$\calV_{\text{ul}}:=K(\calW_1)$$
(the subscript ``ul" stands for ``upper left").
Using the induction hypothesis, we deduce that:
\begin{itemize}
\item[(A)] There exists $Q \in \GL_{n-1}(\K)$ such that $Q\, \calV_{\text{ul}}\, Q^{-1}=\NT_{n-1}(\K)$.
\item[(B)] $C(\calV)=\K^{n-1}$.
\end{itemize}
Setting $P_1:=Q\oplus 1$ and replacing $\calV$ with $P_1 \calV P_1^{-1}$ leaves conditions (A) and (B) unchanged
and does not modify the assumption that $e_n$ is adapted to the space under consideration.
Therefore, we may now assume, in addition to those properties:
\begin{itemize}
\item[(A')] $\calV_{\text{ul}}=\NT_{n-1}(\K)$.
\end{itemize}

\subsection{Corner-compatibility and special matrices in $\calV$}\label{cornercompatsection}

Here, we will repeat part of the strategy of Section \ref{setup}.
Let $M \in \calV$ and assume that $M$ vanishes on $e_2,\dots,e_n$. Then $M \in \calW_1$.
Using $K(M) \in \NT_{n-1}(\K)$, we find $K(M)=0$ and therefore $M=0$. It follows that
$e_1$ is $\calV^T$-adapted.

For any $M$ in $\calV$, we now write:
$$M=\begin{bmatrix}
b(M) & R(M) \\
[?]_{(n-1)\times 1} & I(M)
\end{bmatrix},$$
where $R(M)$ and $I(M)$ are respectively $1 \times (n-1)$ and $(n-1) \times (n-1)$ matrices, and $b(M) \in \K$.
We set
$$\calW_2:=\bigl\{M \in \calV : \; R(M)=0\bigr\},$$
which is a nilpotent linear subspace of the $\K_0$-vector space $\Mat_n(\K)$. Thus $b(M)=0$ for every $M \in \calW_2$,
and $\calV_{\text{lr}}:=I(\calW_2)$ is a nilpotent linear subspace of the $\K_0$-vector space $\Mat_{n-1}(\K)$
(the subscript ``lr" stands for ``lower-right").
Finally, as $e_1$ is $\calV^T$-adapted, we find that
$$\forall M \in \calW_2, \; I(M)=0 \Rightarrow M=0.$$
Using the rank theorem, we deduce that
$$\dim_{\K_0} \calV=\dim_{\K_0} \calV_{\text{lr}}+\dim_{\K_0} R(\calV).$$
As in Section \ref{setup}, equality $\dim_{\K_0} \calV=q\,\binom{n}{2}$ and the induction hypothesis yield:
\begin{itemize}
\item[(C)] There exists $Q' \in \GL_{n-1}(\K)$ such that $\calV_{\text{lr}}=Q'\,\NT_{n-1}(\K)\,(Q')^{-1}$.
\end{itemize}
We aim at modifying $\calV$ once more so as to keep (A') and (B) while sharpening (C).

\begin{Rem}
In the rest of the proof, every matrix of $\Mat_n(\K)$ will be written as a block matrix with the following shape:
$$\begin{bmatrix}
? & [?]_{1 \times (n-2)} & ? \\
[?]_{(n-2) \times 1} & [?]_{(n-2) \times (n-2)} & [?]_{(n-2) \times 1} \\
? & [?]_{1 \times (n-2)} & ?
\end{bmatrix},$$
where the question marks in the corners represent scalars.
\end{Rem}

\vskip 3mm
Let us find some special matrices in $\calV$.
First of all, (A') yields:
\begin{itemize}
\item[(D)] There are $\K_0$-linear mappings
$\varphi : \Mat_{1,n-2}(\K) \rightarrow \Mat_{1,n-2}(\K)$ and $f : \Mat_{1,n-2}(\K) \rightarrow \K$ such that, for
every $L \in \Mat_{1,n-2}(\K)$, the space $\calV$ contains
$$A_L:=\begin{bmatrix}
0 & L & 0 \\
0 & 0 & 0 \\
f(L) & \varphi(L) & 0
\end{bmatrix}.$$
\end{itemize}

\vskip 2mm
Let $C \in \Mat_{n-2,1}(\K)$.
By (B), we know that $\calV$ contains a matrix of the form
$\begin{bmatrix}
? & ? & 0 \\
? & ? & C \\
? & ? & ?
\end{bmatrix}$. By summing it with a matrix of type $A_L$, we may assume furthermore that its first row has the form
$\begin{bmatrix}
? & 0 & \cdots & 0
\end{bmatrix}$: in that case this row is zero as explained above. Therefore, $\calV$ contains a matrix of the following form:
\begin{equation}\label{typeC}\begin{bmatrix}
0 & 0 & 0 \\
? & ? & C \\
? & ? & ?
\end{bmatrix}.
\end{equation}
On the other hand, we know from (A') that, for every $U \in \NT_{n-2}(\K)$, the subspace $\calV$ contains a matrix of the form
\begin{equation}\label{typeU}\begin{bmatrix}
0 & 0 & 0 \\
0 & U & 0 \\
? & ? & 0
\end{bmatrix}.
\end{equation}

We shall now use those observations to prove the following:

\begin{claim}\label{cornercompatclaim}
There exists a row matrix $L \in \Mat_{1,n-2}(\K)$ such that, for $Q_1:=\begin{bmatrix}
I_{n-2} & [0]_{(n-2) \times 1} \\
L & 1
\end{bmatrix}$, one has $Q_1 \,\calV_{\text{lr}}\, Q_1^{-1}=\NT_{n-1}(\K)$.
\end{claim}

\begin{proof}
Let us consider a matrix $Q'$ given by property (C). Denote by $(e_1,\dots,e_{n-1})$ the canonical basis of the $\K$-vector space
$\K^{n-1}$.
Then $\calV_{\text{lr}} x \subset Q'\Vect_\K(e_1,\dots,e_{n-2})$ for every $x \in \K^{n-1}$.
Using the matrices of type \eqref{typeC}, we find that $\calV_{\text{lr}} e_{n-1}$ contains a $q(n-2)$-dimensional subspace of the $\K_0$-vector space
$\K^{n-1}$. Therefore $\calV_{\text{lr}} e_{n-1}=Q'\Vect_\K(e_1,\dots,e_{n-2})$,
and in particular $\calV_{\text{lr}} e_{n-1}$ is an $(n-2)$-dimensional $\K$-linear subspace of $\K^{n-1}$.
Moreover, $\calV_{\text{lr}} e_{n-1}$ has a trivial intersection with $ e_{n-1}\,\K$
since every matrix of $\calV$ is nilpotent. This yields a $\K$-linear map $u : \K^{n-2} \rightarrow \K$ such that
$\calV_{\text{lr}} e_{n-1}=\bigl\{(y,u(y))\mid y \in \K^{n-2}\bigr\}$. Writing $u$ as $(y_1,\dots,y_{n-2}) \mapsto a_1y_1+\cdots+a_{n-2} y_{n-2}$
for some $(a_1,\dots,a_{n-2})\in \K^{n-2}$, we set $L:=\begin{bmatrix}
-a_1 & \cdots & -a_{n-2}
\end{bmatrix}$ and $Q_1:=\begin{bmatrix}
I_{n-2} & [0]_{(n-2) \times 1} \\
L & 1
\end{bmatrix}$. As $\calV_{\text{lr}} x \subset \calV_{\text{lr}} e_{n-1}$ for every $x \in \K^{n-1}$, we deduce that the last row of
every matrix of $\calU:=Q_1 \calV_{\text{lr}} Q_1^{-1}$ is zero.

We now wish to prove that $\calU=\NT_{n-1}(\K)$.
First of all, any matrix $N$ of $\calU$ may be written as
$$N=\begin{bmatrix}
T(N) & [?]_{(n-2) \times 1} \\
[0]_{1 \times (n-2)} & 0
\end{bmatrix} \quad \text{where $T(N)$ is an $(n-2) \times (n-2)$-matrix.}$$
Then $T(\calU)$ is a nilpotent linear subspace of the $\K_0$-vector space $\Mat_{n-2}(\K)$.
With the shape of $Q_1$ and the matrices of type \eqref{typeU}, we find that
$T(\calU)$ contains $\NT_{n-2}(\K)$. As $\dim_{\K_0} T(\calU) \leq q\,\binom{n-2}{2}=\dim_{\K_0} \NT_{n-2}(\K)$
by point (a) in Theorem \ref{Gerstenhaberskew}, we deduce that $T(\calU)=\NT_{n-2}(\K)$.
It follows that $\calU \subset \NT_{n-1}(\K)$, and the equality of dimensions over $\K_0$ then yields $\calU=\NT_{n-1}(\K)$, which finishes the proof.
\end{proof}

With $Q_1$ given by Claim \ref{cornercompatclaim}, we set $P_2:=1 \oplus Q_1$
and replace $\calV$ with $P_2 \calV P_2^{-1}$. Then all the preceding properties are unchanged, but we now have the improved:
\begin{itemize}
\item[(C')] $\calV_{\text{lr}}=\NT_{n-1}(\K)$.
\end{itemize}
Applying that property to the matrices of type \eqref{typeC} and \eqref{typeU}, we find the following properties:

\begin{itemize}
\item[(E)] There is a $\K_0$-linear map $h : \NT_{n-2}(\K) \rightarrow \K$ such that, for every $U \in \NT_{n-2}(\K)$, the space $\calV$
contains the matrix
$$E_U:=\begin{bmatrix}
0 & 0 & 0 \\
0 & U & 0 \\
h(U) & 0 & 0
\end{bmatrix}.$$

\item[(F)] There are two $\K_0$-linear maps $\psi : \Mat_{n-2,1}(\K) \rightarrow \Mat_{n-2,1}(\K)$ and $g : \Mat_{n-2,1}(\K) \rightarrow \K$ such that,
for every $C \in \Mat_{n-2,1}(\K)$, the space $\calV$ contains the matrix
$$B_C:=\begin{bmatrix}
0 & 0 & 0 \\
\psi(C) & 0 & C \\
g(C) & 0 & 0
\end{bmatrix}.$$
\end{itemize}

Finally, for every $a \in \K$, property (B) yields that $\calV$ contains a matrix with entry $a$ at the $(1,n)$-spot:
subtracting matrices of type $A_L$ and $B_C$ from such a matrix yields that $\calV$ contains a matrix of the form
$$J_a=\begin{bmatrix}
? & 0 & a \\
? & ? & 0 \\
? & ? & ?
\end{bmatrix}.$$

\subsection{Analyzing $\varphi$, $\psi$, and performing the last change of basis}\label{analysefetg}

\begin{claim}\label{anafetg1}
For every $L \in \Mat_{1,n-2}(\K)$, there exists $a_L \in \K$ such that $\varphi(L)=a_L\,L$. \\
For every $C \in \Mat_{n-2,1}(\K)$, there exists $b_C \in \K$ such that $\psi(C)=C\,b_C$.
\end{claim}

\begin{proof}
Let $(L,C) \in \Mat_{1,n-2}(\K) \times \Mat_{n-2,1}(\K)$ be such that $LC=0$. \\
Setting $M:=A_L+B_C$, we compute
$$M^2=\begin{bmatrix}
L\psi(C) & 0 & 0 \\
? & ? & 0 \\
? & ? & \varphi(L)C
\end{bmatrix}.$$
As $M \in \calV$, we know that $M^2$ is nilpotent and therefore
$$\varphi(L)C=0 \quad \text{and} \quad L\psi(C)=0.$$
If we fix $L \in \Mat_{1,n-2}(\K)$, varying $C$ yields that the annihilator
of the row matrix $\varphi(L)$ contains that of $L$, and therefore $\varphi(L)=a_L\,L$ for some $a_L \in \K$.
The same line of reasoning yields the second part of Claim \ref{anafetg1}.
\end{proof}

\begin{claim}\label{anafetg2}
There is a scalar $\lambda \in \K$ such that
$$\forall (L,C)\in \Mat_{1,n-2}(\K) \times \Mat_{n-2,1}(\K), \quad \varphi(L)=\lambda\, L \quad \text{and} \quad \psi(C)=-C\,\lambda.$$
\end{claim}

\begin{proof}
By Claim \ref{anafetg1}, there are endomorphisms $\varphi_1,\dots,\varphi_{n-2}$ of the $\K_0$-vector space $\K$ such that
$$\forall L=\begin{bmatrix}
l_1 & \cdots & l_{n-2}
\end{bmatrix} \in \Mat_{1,n-2}(\K), \; \varphi(L)=\begin{bmatrix}
\varphi_1(l_1) & \cdots & \varphi_{n-2}(l_{n-2})
\end{bmatrix}.$$
Applying Claim \ref{anafetg1} to the row matrices in which all the entries are equal, we find $\varphi_1=\cdots=\varphi_{n-2}$.
As the same line of reasoning applies to $\psi$, we obtain two endomorphisms $u$ and $v$ of the $\K_0$-vector space $\K$ such that
$$\forall L=\begin{bmatrix}
l_1 & \cdots & l_{n-2}
\end{bmatrix} \in \Mat_{1,n-2}(\K), \quad \varphi(L)=\begin{bmatrix}
u(l_1) & \cdots & u(l_{n-2})
\end{bmatrix}$$
and
$$\forall C=\begin{bmatrix}
c_1 & \cdots & c_{n-2}
\end{bmatrix}^T \in \Mat_{n-2,1}(\K), \quad \psi(C)=\begin{bmatrix}
v(c_1) & \cdots & v(c_{n-2})
\end{bmatrix}^T.$$

Let $(a,b) \in \K^2$, and set $L_0:=\begin{bmatrix}
a & 0 & \cdots & 0
\end{bmatrix} \in \Mat_{1,n-2}(\K)$ and $C_0:=\begin{bmatrix}
b & 0 & \cdots & 0
\end{bmatrix}^T \in \Mat_{n-2,1}(\K)$.
We notice that $M:=A_{L_0}+B_{C_0}$ stabilizes the $\K$-subspace $\Vect_\K(e_1,e_2,e_n)$ and induces an endomorphism of it
represented by $N=\begin{bmatrix}
0 & a & 0 \\
v(b) & 0 & b \\
? & u(a) & 0
\end{bmatrix}$. Then $N$ is a $3 \times 3$ nilpotent matrix, and therefore $N^3=0$.
One computes that the entry of $N^3$ at the $(1,2)$-spot is
$a\bigl(v(b)a+bu(a)\bigr)$. For $a \neq 0$, this yields
\begin{equation}\label{uetv}
v(b)\,a+b\,u(a)=0,
\end{equation}
which is also obviously true for $a=0$. \\
Set now $\lambda:=u(1)$. Taking $a=1$ in \eqref{uetv} yields: $v(b)=-b\,\lambda$ for all $b \in \K$.
Thus, $v(1)=-\lambda$, and taking $b=1$ in \eqref{uetv} yields $u(a)=\lambda\, a$ for all $a \in \K$.
This finishes the proof of Claim \ref{anafetg2}.
\end{proof}

\begin{Rem}
In the case $\K=\K_0$, Claim \ref{anafetg2} has a far more simple proof. Indeed,
Claim \ref{anafetg1} then readily yields a pair $(\lambda,\mu) \in \K^2$ such that
$\forall (L,C) \in \Mat_{1,n-2}(\K) \times \Mat_{n-2,1}(\K), \; \varphi(L)=\lambda\,L \quad \text{and} \quad \psi(C)=\mu\, C$;
as $\K$ is commutative, we find $\tr(A_LB_C)=0$ for every $(L,C) \in \Mat_{1,n-2}(\K) \times \Mat_{n-2,1}(\K)$, and hence $\mu+\lambda=0$.
\end{Rem}

Now, we perform one last change of basis. We set $P:=\begin{bmatrix}
1 & 0 & 0 \\
0 & I_{n-2} & 0 \\
-\lambda & 0 & 1
\end{bmatrix} \in \GL_n(\K)$ and we replace $\calV$ with $P\calV P^{-1}$. Note then that all properties
(A'), (B), (C'), (D), (E) and (F) still hold, but we now have a simplified form for the matrices of type $A_L$ and $B_C$:
$$\forall (L,C) \in \Mat_{1,n-2}(\K) \times \Mat_{n-2,1}(\K), \;
A_L=\begin{bmatrix}
0 & L & 0 \\
0 & 0 & 0 \\
f(L) & 0 & 0
\end{bmatrix} \; \text{and} \;
B_C=\begin{bmatrix}
0 & 0 & 0 \\
0 & 0 & C \\
g(C) & 0 & 0
\end{bmatrix}.$$
From there, our aim is to prove that $\calV=\NT_n(\K)$.
In order to do so, we will show that all the matrices of type $A_L$, $B_C$, $E_U$ and $J_a$
are strictly upper-triangular. This will prove the inclusion $\NT_n(\K) \subset \calV$,
and the equality of dimensions over $\K_0$ will help us complete the proof.
We start by showing that $f$ and $g$ vanish everywhere.

\subsection{The vanishing of $f$ and $g$}

\begin{claim}\label{phipsiclaim}
One has $f=0$ and $g=0$.
\end{claim}

\begin{proof}
We claim that
\begin{equation}\label{phietpsi}
\forall (L,C) \in \Mat_{1,n-2}(\K) \times \Mat_{n-2,1}(\K), \; LC \neq 0 \Rightarrow f(L)+g(C)=0.
\end{equation}
Let indeed $(L,C) \in \Mat_{1,n-2}(\K) \times \Mat_{n-2,1}(\K)$ be such that $LC \neq 0$;
setting $M:=A_L+B_C$, we compute $M^3 e_1=e_1\,\bigl(LC(f(L)+g(C))\bigr)$ and \eqref{phietpsi} follows as $M^3$ is nilpotent.
\begin{itemize}
\item Assume that $n \geq 4$. Let $L \in \Mat_{1,n-2}(\K)$. As $n-2 \geq 2$, we may choose $C \in \Mat_{n-2,1}(\K) \setminus \{0\}$ such that $LC=0$,
and then we may choose $L_1 \in \Mat_{1,n-2}(\K)$ such that $L_1C=1$. Then $(L+L_1)C=1$,
which yields $f(L+L_1)=-g(C)=f(L_1)$. Thus, $f(L)=0$. The same line of reasoning yields $g=0$.
\item Assume that $n=3$ and $\# \K>2$. Let $x \in \K$. Then we may choose $y \in \K \setminus \{0,-x\}$, so that $y \neq 0$ and $x+y \neq 0$.
Therefore, $f(x+y)=-g(1)=f(y)$, and hence $f(x)=0$. The same line of reasoning yields $g=0$.
\item Assume finally that $n=3$ and $\# \K=2$, so that $\K_0=\K \simeq \F_2$. Then, $f(1)=g(1)$. Assume that $f(1)=1$.
Then $\calV$ contains the matrices
$$A:=\begin{bmatrix}
0 & 1 & 0 \\
0 & 0 & 0 \\
1 & 0 & 0
\end{bmatrix} \quad \text{and} \quad B:=\begin{bmatrix}
0 & 0 & 0 \\
0 & 0 & 1 \\
1 & 0 & 0
\end{bmatrix}$$
and a matrix of the form
$$J=\begin{bmatrix}
a & 0 & 1 \\
b & c & 0 \\
d & e & f
\end{bmatrix}.$$
Note that $\K$ is commutative, thus Lemma \ref{orthogonality} yields $\tr(AJ)=\tr(BJ)=0$, and hence $b=e=1$. As $J$ is nilpotent, we also have
$\tr(J)=0$, and hence $f=a+c$.
Using $\forall t \in \K, \; t^2=t \; \text{and} \; 2t=0$, we finally compute:
$$\forall (x,y)\in \K^2, \; 0=\det(J+xA+yB)=1+cd+(a+c)\,y+a\,x+d\,xy.$$
This yields both $cd=1$ and $d=0$, a contradiction.

Therefore, $f(1)=g(1)=0$, and so  $f=0$ and $g=0$, as claimed.
\end{itemize}
\end{proof}

\subsection{The vanishing of $h$}

\begin{claim}
One has $h=0$.
\end{claim}

\begin{proof}
Let $U \in \NT_{n-2}(\K)$ be such that $U^2=0$.
Set $L_0:=\begin{bmatrix}
1 & 0 & \cdots & 0
\end{bmatrix} \in \Mat_{1,n-2}(\K)$ and $C_0:=L_0^T$, so that $L_0UC_0=0$ and $L_0C_0=1$.
Setting $M:=A_{L_0}+B_{C_0}+E_U$, one checks that $M^3 e_n=e_n\,h(U)$, and therefore
$h(U)=0$. \\
In particular, $h(E_{i,j}\,a)=0$ for every $a \in \K$ and every $(i,j)\in \lcro 1,n-2\rcro^2$ with $j>i$
(where $E_{i,j}$ is the matrix with all entries zero except at the $(i,j)$-spot where the entry is $1$).
As $h$ is additive, we deduce that $h$ vanishes everywhere on $\NT_{n-2}(\K)$.
\end{proof}

\subsection{The matrices of type $J_a$}

\subsubsection{Simplifying the $J_a$ matrices}

Let us sum up. For every triple $(L,C,U) \in \Mat_{1,n-2}(\K) \times \Mat_{n-2,1}(\K) \times \NT_{n-2}(\K)$, the space
$\calV$ contains the matrices
$$A_L=\begin{bmatrix}
0 & L & 0 \\
0 & 0 & 0 \\
0 & 0 & 0
\end{bmatrix}, \quad
B_C=\begin{bmatrix}
0 & 0 & 0 \\
0 & 0 & C \\
0 & 0 & 0
\end{bmatrix} \quad \text{and} \quad
E_U=\begin{bmatrix}
0 & 0 & 0 \\
0 & U & 0 \\
0 & 0 & 0
\end{bmatrix}.$$

Adding an appropriate $E_U$ to each matrix of type $J_a$, one finds $\K_0$-linear maps
$\alpha : \K \rightarrow \K$, $\beta : \K \rightarrow \K$, $\gamma : \K \rightarrow \K$,
$L_1 : \K \rightarrow \Mat_{1,n-2}(\K)$, $C_1 : \K \rightarrow \Mat_{n-2,1}(\K)$, $T : \K \rightarrow \LT_{n-2}(\K)$
(where $\LT_{n-2}(\K)$ denotes the set of lower-triangular matrices of $\Mat_{n-2}(\K)$)
such that, for every $a \in \K$, the subspace $\calV$ contains
$$J_a:=\begin{bmatrix}
\alpha(a) & 0 & a \\
C_1(a) & T(a) & 0 \\
\beta(a) & L_1(a) & \gamma(a)
\end{bmatrix}.$$
Our aim in what follows is to prove:

\begin{claim}\label{lastclaim}
All the maps $\alpha$, $\beta$, $\gamma$, $L_1$, $C_1$ and $T$ vanish everywhere on $\K$.
\end{claim}

We have to distinguish between two cases, the main problem being the handling of fields with two elements.

\subsubsection{Proof of Claim \ref{lastclaim}: the case $\K=\K_0$}

We assume $\K=\K_0$. In particular, $\K$ is commutative, which allows us to use
Lemma \ref{orthogonality} to obtain $\tr(J_1 A_L)=0$, $\tr(J_1B_C)=0$ and $\tr(J_1E_U)=0$ for all
$(L,C,U) \in \Mat_{1,n-2}(\K) \times \Mat_{n-2,1}(\K) \times \NT_{n-2}(\K)$. Therefore, $L_1(1)=0$, $C_1(1)=0$ and $T(1)$ is a diagonal matrix.
Every diagonal entry of $T(1)$ is an eigenvalue of $J_1$, and hence $T(1)=0$.
Then $J_1$ induces an endomorphism of $\Vect_\K(e_1,e_n)$ whose matrix in $(e_1,e_n)$
is $N=\begin{bmatrix}
\alpha(1) & 1 \\
\beta(1) & \gamma(1)
\end{bmatrix}$. This last matrix must be nilpotent, and hence $\alpha(1)=-\gamma(1)$ and $\beta(1)=-\gamma(1)^2$ (as $\tr N=0$ and $\det N=0$).
Choose finally $(L,C)\in \Mat_{1,n-2}(\K) \times \Mat_{n-2,1}(\K)$ such that $LC \neq 0$, and set
$M:=J_1+A_L+B_C$. One checks that $M^3e_1=-\gamma(1)^2 LC \,e_1$, and hence $\gamma(1)=0$. Therefore, the maps
$\alpha$, $\beta$, $\gamma$, $L_1$, $C_1$ and $T$ all vanish on $1$; since they are $\K$-linear, Claim \ref{lastclaim} is proven in the case
$\K=\K_0$.

\subsubsection{Proof of Claim \ref{lastclaim}: the case $\# \K>2$}

We assume here that $\# \K>2$, which holds whenever $\K_0 \subsetneq \K$.

Fix $a \in \K$.
Let $C_0 \in \Mat_{n-2,1}(\K) \setminus \{0\}$. Let $x \in \K$.
We consider the non-zero vector $X:=\begin{bmatrix}
x \\
C_0 \\
1
\end{bmatrix}$ of $\K^n$. The $\K_0$-vector space $\calV X$ must intersect $X\,\K$ trivially as all the elements of $\calV$ are nilpotent.
Thus $\dim_{\K_0} \calV X \leq (n-1)q$.
However, for every $(L,C) \in \Mat_{1,n-2}(\K) \times \Mat_{n-2,1}(\K)$, we have
$$A_L X=\begin{bmatrix}
LC_0 \\
0 \\
0
\end{bmatrix} \quad \text{and} \quad B_C X=\begin{bmatrix}
0 \\
C \\
0
\end{bmatrix}$$
Varying $L$ and $C$ then yields the inclusion $\K^{n-1} \times \{0\} \subset
\calV X$. Since $\dim_{\K_0} \calV X \leq (n-1)q=\dim_{\K_0} (\K^{n-1} \times \{0\})$, we deduce that $\calV X=\K^{n-1} \times \{0\}$.
However, the last entry of $J_a X$ is $\beta(a) x+L_1(a)C_0+\gamma(a)$, and therefore:
$$\forall x\in \K, \; \; \beta(a) x+L_1(a)C_0+\gamma(a)=0.$$
We deduce that $L_1(a)C_0+\gamma(a)=0$ and $\beta(a)=0$, which yields:
$$\forall C \in \Mat_{n-2,1}(\K) \setminus \{0\}, \; \forall y \in \K \setminus \{0\}, \; L_1(a)Cy+\gamma(a)=0.$$
As $\# \K>2$, we deduce that $\gamma(a)=0$ and
$$\forall C \in \Mat_{n-2,1}(\K) \setminus \{0\}, \; L_1(a)C=0.$$
Varying $C$ then yields $L_1(a)=0$.

Let again $C_0 \in \Mat_{1,n-2}(\K) \setminus \{0\}$, and set $Y:=\begin{bmatrix}
1 \\
C_0 \\
0
\end{bmatrix}$.
For every $(L,C) \in \Mat_{1,n-2}(\K) \times \Mat_{n-2,1}(\K)$, we have
$$A_L^T Y=\begin{bmatrix}
0 \\
L^T \\
0
\end{bmatrix} \quad \text{and} \quad B_C^T X=\begin{bmatrix}
0 \\
0 \\
C^TC_0
\end{bmatrix}.$$
As above, varying $C$ and $L$ yields $\calV^T Y=\{0\} \times \K^{n-1}$.
The first entry of $J_a^T Y$ is $\alpha(a)+C_1(a)^TC_0$ and it must be $0$.
Again, varying $C_0$ yields both $\alpha(a)=0$ and $C_1(a)=0$.

Let $U \in \NT_{n-2}(\K)$.
For every $t \in \K_0$, the matrix $E_U+tJ_a$ is nilpotent and stabilizes the $\K$-vector space $\Vect_\K(e_2,\dots,e_{n-1})$,
with an induced endomorphism represented in $(e_2,\dots,e_{n-1})$ by $U+t\,T(a)$. It follows that $\NT_{n-2}(\K)+\K_0 T(a)$
is a nilpotent $\K_0$-linear subspace of $\Mat_{n-2}(\K)$. If $T(a) \neq 0$, then we have a contradiction with
point (a) of Theorem \ref{Gerstenhaberskew}. Therefore $T(a)=0$, and Claim \ref{lastclaim} is proven.

\subsection{Conclusion}

We have shown that, for every list $(L,C,U,a) \in \Mat_{1,n-2}(\K) \times \Mat_{n-2,1}(\K) \times \NT_{n-2}(\K) \times \K$, the additive group
$\calV$ contains all four matrices
$$\begin{bmatrix}
0 & L & 0 \\
0 & 0 & 0 \\
0 & 0 & 0
\end{bmatrix}, \quad
\begin{bmatrix}
0 & 0 & 0 \\
0 & 0 & C \\
0 & 0 & 0
\end{bmatrix}, \quad
\begin{bmatrix}
0 & 0 & 0 \\
0 & U & 0 \\
0 & 0 & 0
\end{bmatrix} \quad \text{and} \quad
\begin{bmatrix}
0 & 0 & a \\
0 & 0 & 0 \\
0 & 0 & 0
\end{bmatrix}.$$
It follows that $\calV$ contains $\NT_n(\K)$. As $\dim_{\K_0} \calV=q\,\binom{n}{2}=\dim_{\K_0} \NT_n(\K)$, we conclude that
$\calV=\NT_n(\K)$. This completes our proof of point (b) of Theorem \ref{Gerstenhaberskew}.

\section*{Acknowledgement}

The author would like to thank Alexander Guterman for his outstanding effort in helping him understand the essence of Serezhkin's proof.

\end{document}